\documentclass[oneside, 12pt, a4paper]{amsart}
\usepackage{amsmath}
\usepackage{amssymb}
\usepackage{amsthm}
\usepackage{graphicx}
\usepackage{caption}
\usepackage{tikz}

\usepackage{pict2e}   
\usepackage{epic}
\usepackage{xcolor}
\newcommand{\bn}[2]{\put(#1,#2){\circle*{4}}}
\newcommand{\inb}[2]{\put(#1,#2){\circle*{4}}\put(#1,#2){\circle{13}}}
\newcommand{\elx}[1]{\put(#1,0){\makebox(0,0){$\cdots$}}}
\newsavebox{\dynbox}
\newenvironment{dyn}[4]%
  {\begin{lrbox}{\dynbox}\begin{picture}(#1,#2)(#3,#4)}%
  {\end{picture}\end{lrbox}\ensuremath{\vcenter{\copy\dynbox}}}

\usepackage{amsmath,amssymb,amsfonts}
\usepackage{sseq}
\usepackage{dsfont}
\usepackage{mathtools}
\usepackage{enumerate}
\usepackage{enumitem}
\usepackage{doi}
\usepackage{faktor}
\usepackage[font=small,labelfont=bf]{caption}
\usepackage{eucal}
\usepackage{mathrsfs}
\usepackage{stmaryrd}
\usepackage{verbatim}
\usepackage{manfnt}
\usepackage{array}
\usetikzlibrary{patterns,cd}
\usepackage{layout} 
\usepackage[all]{xy}

\usepackage{faktor}
\usepackage{lua-visual-debug}

\numberwithin{equation}{section}

\let\latexbegin\begin

\def\amsamp{&}

\begingroup
\catcode`\&=13
\gdef\pampmatrix{%
  \begingroup
  \let&\amsamp
  \latexbegin{pmatrix}%
}

\endgroup

\newcommand{\R}{\mathbb{R}}
\newcommand{\F}{\mathbb{F}}
\newcommand{\Z}{\mathbb{Z}}
\newcommand{\C}{\mathbb{C}}

\newcommand{\Q}{\mathbb{Q}}

\newcommand{\Hom}{\operatorname{Hom}}

\newcommand{\co}{\colon\thinspace}

\theoremstyle{plain}

\newtheorem{theorem}{Theorem}[section]

\newtheorem*{theorem*}{Theorem}

\newtheorem{proposition}[theorem]{Proposition}

\newtheorem{lemma}[theorem]{Lemma}

\newtheorem{corollary}[theorem]{Corollary}

\theoremstyle{definition}

\newtheorem*{open*}{Open problem}
\newtheorem*{observation*}{Observation}
\newtheorem*{question*}{Question}
\newtheorem*{task*}{Task}

\newtheorem*{notation*}{Notation}

\newcounter{exercises}
\theoremstyle{remark}
\newtheorem{remark}[theorem]{Remark}

\usepackage[all]{xy}
\SelectTips{cm}{10}
\CompileMatrices

\DeclareMathAlphabet{\mathpzc}{OT1}{pzc}{m}{it}

\title[Torsion primes for spaces of commuting elements]{Torsion primes for spaces of commuting elements in Lie groups}

\author{Simon Gritschacher}
\address{Mathematical Institute,
Ludwig-Maximilians-University,
Theresienstra{\ss}e 39,
80333 Munich,
Germany}
\email{simon.gritschacher@math.lmu.de}
\author{Chris Hone}
\address{Department of Mathematical Sciences,
University of Copenhagen,
Universitetsparken 5,
DK-2100 Copenhagen \O,
Denmark}
\email{cth@math.ku.dk}
\date{\today}

\begin{document}

\maketitle

\begin{abstract}
Let $G$ be a complex reductive group and $C_m(G)_1$ the identity component of the space of $m$-tuples of commuting elements in $G$. We prove that for every $m\geq 2$, the integral singular cohomology of $C_m(G)_1$ has torsion precisely at those primes which divide the order of the Weyl group of $G$. We deduce the analogous statement for compact Lie groups, settling a conjecture of Kishimoto and Takeda, and prove the corresponding result for compact Lie algebras. We prove these results using Smith theory for the conjugation action by semisimple elements of prime power order. \end{abstract}


\section{Introduction}
For a topological group $G$ and an integer $m\geq 1$ we consider the space of $m$-tuples of commuting elements in $G$:
\[
C_m(G)=\{(g_1,\dots,g_m)\in G^m\mid g_ig_j=g_jg_i \textnormal{ for } 1\leq i,j\leq m\}\subseteq G^m\, .
\]
We denote by $C_m(G)_1$ the connected component containing the tuple $(1,\dots,1)$. In this paper we characterise the torsion occurring in the integral cohomology of this space for various topological groups. The first case is that of the complex points of a reductive algebraic group. 

\begin{theorem} \label{thm:mainreductive}
Let $G$ be a complex reductive group, with Weyl group $W$. For $m\geq 2$ and prime $p$, the singular cohomology $H^*(C_m(G)_1,\Z)$ has nontrivial $p$-torsion if and only if $p$ divides the order of $W$.
\end{theorem}

Any complex reductive group has a maximal compact Lie subgroup $K$, and any compact Lie group arises in this way. The homotopical properties of these groups are intimately related; the inclusion $K\to G$ is a homotopy equivalence, and by results of Pettet-Souto \cite{PettetSouto}, the corresponding map $C_m(K)_1\to C_m(G)_1$ is also a homotopy equivalence. We therefore deduce the corresponding Lie theoretic statement:

\begin{theorem} \label{thm:mainlie}
For a compact Lie group $K$ and any $m\geq 2$, the torsion occurring in the integral cohomology of $C_m(K)_1$ is precisely at primes dividing the order of the Weyl group of $K$.
\end{theorem}

This extends the partial results of Kishimoto-Takeda \cite{KishimotoTakeda1} and proves their conjecture \cite[Conjecture 7.11]{KishimotoTakeda1}. It should be noted that Theorem \ref{thm:mainreductive} also holds for $G$ the group of real points of a reductive algebraic group defined over $\R$, and $W$ the Weyl group of its maximal compact subgroup $K\subset G$. This follows at once from Theorem \ref{thm:mainlie} and the  real version of the homotopy equivalence of Pettet-Souto \cite{PettetSouto}.

Our methods also yield the corresponding statement for the compactly supported cohomology of the space of commuting $m$-tuples in a compact real Lie algebra $\mathfrak{g}$:
\[C_m(\mathfrak{g})=\{(X_1,\dots,X_m)\in \mathfrak{g}^m\mid [X_i,X_j]=0 \textnormal{ for } 1\leq i,j\leq m\}\subseteq \mathfrak{g}^m\, .\]

\begin{theorem} \label{thm:mainliealgebra}
Let $\mathfrak{g}$ be a compact real Lie algebra. Then, for all $m\geq 2$, the integral compactly supported cohomology of $C_m(\mathfrak{g})$ has torsion precisely at the primes dividing the order of the Weyl group of $\mathfrak{g}$.
\end{theorem}

These methods break down, however, for spaces of commuting tuples in a complex reductive Lie algebra, see Proposition \ref{prop:complexreductiveliealgebra}.

For the spaces $C_m(G)_1$, it is well known that the only torsion potentially occurring is at primes dividing the Weyl group order, so the difficulty lies in showing that enough torsion actually occurs.

In essence, our proof of Theorem \ref{thm:mainreductive} is an induction over the dimension of $G$, where in each inductive step we pass to a suitable centraliser subgroup and our control over $p$-torsion comes from Smith theory. More precisely, our strategy is based on the following facts, where (2) is specific to the spaces $C_m(G)_1$:

\begin{enumerate}
    \item The space $C_m(G)_1$ has $p$-torsion in its integral cohomology if and only if the dimension of its total $\F_p$-cohomology is larger than that of its total rational cohomology.
    \item For any semisimple element $x$ in $G$, with identity component of its centraliser $Z_G(x)_0$, the dimensions of the total rational cohomologies agree:
    \[\dim_\Q H^*(C_m(Z_G(x)_0)_1,\Q)=\dim_\Q H^*(C_m(G)_1,\Q)\, .\]
    
    \item For an order $p^k$ automorphism of $C_m(G)_1$, the dimension of the total $\F_p$-cohomology of its fixed points is at most that of the total $\F_p$-cohomology of $C_m(G)_1$. Therefore, for conjugation by a semisimple element $x$ of $G$ of order $p^k$, we have: 
    \[\dim_{\F_p}H^*(C_m(Z_G(x)_0)_1,\F_p)\leq \dim_{\F_p} H^*(C_m(G)_1,\F_p)\, .\]
\end{enumerate}

With these, we reduce to the situation where $G'$ is simple of type $A_{p-1}$, and we exhibit torsion in this case explicitly. Though this result for $SU(p)$ is known due to \cite{KishimotoTakeda1}, we use a different method, as the degeneration of the spectral sequence (Proposition \ref{prop:ss}) may be of independent interest.

Our results contrast with the following classical torsion characterisation of Borel \cite{Borel}, that for simply connected $G$, $p$-torsion occurs in the cohomology of $C_1(G)=G$ if and only if $p$ divides a coroot integer of $G$. In particular, one has no torsion in type $A$. Our argument shows that not only is there plenty of torsion in the cohomology of $C_m(G)_1$ for $m\geq 2$, but also that it all has type $A$ origin. 

\section*{Acknowledgments}

The second author is supported by the Danish National Research Foundation through the Copenhagen Centre for Geometry and Topology (DNRF151). This work was initiated during the workshop ``Cohomology of Finite Groups: Interactions and Applications'' (workshop no. 2609) at the Mathematisches Forschungsinstitut Oberwolfach. The authors gratefully acknowledge the institute's hospitality and support.

\section*{AI Disclosure}

Generative AI tools were not used to generate the ideas or text within this paper, and were used for editing purposes only.

\section{The space of commuting elements}

In this section we recall some topological properties of the spaces $C_m(G)_1$ that we will require. In the introduction we have already mentioned the following theorem of Pettet-Souto \cite{PettetSouto}:

\begin{theorem} \label{thm:pettetsouto}
    Let $G$ be the group of complex or real points of a reductive algebraic group, defined over $\R$ in the latter case. For $K\subset G$ a maximal compact subgroup, the inclusion \[C_m(K)_1\to C_m(G)_1\] admits a deformation retraction. In particular, it is a homotopy equivalence.
\end{theorem}

This theorem allows one to freely work in the algebraic or compact Lie group context when analysing the spaces $C_m(G)_1$. 

From here on, and if not said otherwise, $G$ will denote a complex reductive group, and $K$ a compact Lie group. Notice that $C_m(G)_1$ automatically picks out the identity component $G_0\subset G$ (and similarly for $C_m(K)_1$), and so we may assume without loss of generality that our groups are connected.

From a point set perspective, the spaces $C_m(G)_1$ and $C_m(K)_1$ are nice, being semialgebraic sets. They are metrisable, hence paracompact Hausdorff, finitistic, of finite covering dimension, and are homotopy equivalent to finite CW complexes. They are also locally contractible, so \v{C}ech and singular cohomology agree, and the total rank of their cohomology is finite over any field.

The spaces $C_m(G)_1$ are typically singular, though in the case of compact $K$, the space $C_m(K)_1$ has a canonical desingularisation, induced by conjugating commuting $m$-tuples lying in a fixed maximal torus $T$. The following result is due to Baird \cite{Baird}:

\begin{proposition}\label{prop:desing}
The conjugation map $K\times T^m\to C_m(K)_1$ induces a proper surjective map from the smooth manifold quotient: \[K/T\times_W T^m\to C_m(K)_1.\] This map is a homeomorphism over the dense subset of $C_m(K)_1$ of $m$-tuples containing a regular element, and it induces an isomorphism in cohomology with coefficients in a field $k$ of characteristic coprime to $|W|$.
\end{proposition}

\begin{remark}
Although $K/T$ is diffeomorphic to the complex projective flag variety $G/B$, the action by $W$ is not algebraic; it does not preserve the complex structure, as the induced representation on top cohomology is the sign representation of $W$.
\end{remark}

Ramras and Stafa observed in \cite[Corollary 5.2]{RamrasStafa} that the total rank of the rational cohomology of $C_m(K)_1$ is $2^{mr}$, where $r$ is the rank of $K$. So, it depends only on $m$ and on the dimension of a maximal torus.

\begin{corollary} \label{cor:constantrank}
Let $G$ have rank $r$, and let $k$ be a field of characteristic coprime to $|W|$. Then the total dimension of  $H^*(C_m(G)_1,k)$ is $2^{mr}$. For $x$ semisimple in $G$, with identity component of the centraliser $Z_G(x)_0$, we therefore have an equality of total dimensions:
\[\dim_k H^*(C_m(G)_1,k)=\dim_k H^*(C_m(Z_G(x)_0)_1,k)=2^{mr}.\]
\end{corollary}

\begin{proof}
By Pettet-Souto, we reduce at once to the case of compact $K$ with maximal torus $T$. The $kW$-module $H^*(K/T,k)$ is free of rank one, so using the theorem of Baird we have: 
\begin{align*}
    \dim_k H^*(C_m(K)_1,k)&=
    \dim_k H^*(K/T\times_W T^m,k)\\&=\dim_k H^*(K/T\times T^m,k)^W\\
    &=\frac{1}{|W|}\dim_k H^*(K/T,k)\cdot \dim_k H^*(T^m,k)\\
    &=2^{mr}.
\end{align*}

For $x$ a semisimple element in $G$, the identity component of its centraliser $Z_G(x)_0$ is a connected reductive group with the same rank as $G$, from which the second claim follows.
\end{proof}

\begin{remark}
Though this dimension only depends on the rank of $G$, the Betti numbers of $C_m(G)_1$ depend on the root system, see \cite[Theorem 1.1]{RamrasStafa}.
\end{remark}

\section{Torsion via Smith theory}
In this section we describe our inductive approach for deducing $p$-torsion in the cohomology of $C_m(G)_1$ from torsion in the corresponding space for centraliser subgroups.

To detect torsion in integral cohomology we use the following lemma.

\begin{lemma} \label{lem:recognition}
If $X$ has the homotopy type of a finite CW-complex, then $X$ has $p$-torsion in its integral cohomology if and only if \[\dim_{\mathbb{Q}}H^*(X,\mathbb{Q})< \dim_{\mathbb{F}_p}H^*(X,\mathbb{F}_p)\]
\end{lemma}
\begin{proof}
This follows at once from the universal coefficient theorem.
\end{proof}

This then yields the restriction on potential $p$-torsion.

\begin{corollary} \label{cor:modular}
If $p$-torsion occurs in the cohomology of $C_m(G)_1$, then $p$ divides the order of $W$, the Weyl group of $G$.
\end{corollary}

\begin{proof}
If $p$ does not divide the order of $W$, Corollary \ref{cor:constantrank} shows that the total ranks of the rational and $\mathbb{F}_p$-cohomology agree.
\end{proof}

So the problem is to show the existence of $p$-torsion for all primes dividing $W$. We will use a theorem of Floyd \cite[Theorem 4.4]{Floyd}; for the cyclic case see also \cite[Theorem 3.10.4]{AlldayPuppe}.
\begin{theorem}\label{thm:floyd}
Let $X$ be a paracompact Hausdorff, finitistic space, $p$ a prime, and assume $\dim_{\F_p}\check{H}^\ast(X,\F_p) < \infty$. Then for $P$ a finite $p$-group acting on $X$, we have:
\[
\dim_{\F_p} \check{H}^\ast(X^P,\F_p)\, \leq \dim_{\F_p} \check{H}^\ast(X,\F_p). 
\]
\end{theorem}

Here, $\check{H}$ is \v{C}ech cohomology.

\begin{proof}
The cyclic case $P=C_p$ is \cite[Theorem 3.10.4]{AlldayPuppe}. The general case follows by induction over the order of $P$, by choosing a (normal) subgroup $P'\leq P$ with $P/P'\cong C_p$ and using $(X^{P'})^{C_p}= X^P$.
\end{proof}

Our spaces satisfy the conditions of Floyd's result, enabling the following key inductive step.

\begin{proposition} \label{prop:smith}
Let $x$ in $G$ be semisimple of order a power of $p$. Then, if $C_m(Z_G(x)_0)_1$ has $p$-torsion in its cohomology, so does $C_m(G)_1$.
\end{proposition}
\begin{proof}
These spaces have the same total dimension of rational cohomology by Corollary \ref{cor:constantrank}, and the fixed points of $x$ acting on $C_m(G)_1$ by conjugation have $C_m(Z_G(x)_0)_1$ as a connected component. Coupled with Floyd's inequality, Theorem \ref{thm:floyd}, we obtain:
\begin{equation*}
\begin{split}
\dim_{\mathbb{Q}}H^*(C_m(G)_1,\mathbb{Q}) &= \dim_{\mathbb{Q}}H^*(C_m(Z_G(x)_0)_1,\mathbb{Q}) \\&<
\dim_{\mathbb{F}_p}H^*(C_m(Z_G(x)_0)_1,\mathbb{F}_p)\\&\leq \dim_{\mathbb{F}_p}H^*((C_m(G)_1)^x,\mathbb{F}_p)\\&\leq \dim_{\mathbb{F}_p}H^*(C_m(G)_1,\mathbb{F}_p)\, .
\end{split}
\end{equation*}
Now the result follows from Lemma \ref{lem:recognition}.
\end{proof}

\section{Reduction to $A_{p-1}$} \label{sec:reduction}

In this section we reduce the general case of Theorem \ref{thm:mainreductive} to the case where $G$ is $GL_p(\C),SL_p(\C)$ or $PGL_p(\C)$. Here and below we take $p$ to be a fixed prime dividing $|W|$. In view of Proposition \ref{prop:smith}, if $G$ is not yet one of these three options, we need to find a semisimple element in $G$ of order $p^k$ whose centraliser has strictly smaller dimension, while preserving the fact that $p$ divides the order of the Weyl group.

First we reduce to the case where the derived subgroup $G'$ is simple.

\begin{proposition} \label{prop:simple}
Let $p$ be a prime dividing the order of $W$. Then there is a semisimple $x$ in $G$ of order $p^k$ for some $k\geq 0$ such that $Z_G(x)_0'$ is simple and $p$ divides the order of the Weyl group of $Z_G(x)_0$.
\end{proposition}

\begin{proof}
There is an isogeny $f\co Z(G)_0\times \prod_{i=1}^n G_i\to G$ with $G_i$ simple for all $i$, $W=\prod_{i=1}^n W(G_i)$ and $p$ divides the order of $W(G_j)$ for some $j$. Choose regular semisimple elements $y_i$ of order $p^{k_i}$ in $G_i$ for $i\neq j$ and set $y:=(1,y_1,\dots,y_{j-1},1,y_{j+1},\dots,y_n)$. Setting $x:=f(y)$ in $G$, we have that $Z_G(x)_0$ is a finite central quotient of a torus times $G_j$. For this $x$ we then see that $Z_G(x)'_0$ is simple, and $p$ divides the order of the Weyl group of $Z_G(x)_0$.
\end{proof}

In view of Proposition \ref{prop:smith}, we may from now on assume that $G'$ is simple. Next, we describe a suitable family of semisimple elements $x$ in $G'$, using the root system.

Let $H$ be a simple complex algebraic group (below, we will consider $H=G'$). Let $T\subset H$ be a maximal torus and
\[
\Phi=\Phi(H,T)\subset \Hom(T,\C^\times)
\]
the set of roots for $H$. Let $\alpha_1,\dots,\alpha_r$ be a basis of simple roots for $\Phi$, $\omega_1^\vee,\dots,\omega_r^\vee$ the fundamental coweights defined by $(\alpha_j,\omega_i^\vee)=\delta_{ij}$, and $X_\ast(T)=\Hom(\C^\times,T)$ the cocharacter lattice. Since $X_\ast(T)$ has finite index in the coweight lattice, we can fix for each $i=1,\dots,r$ a non-zero integer $d_i$ such that
\[
u_i:=d_i \omega_i^\vee\in X_\ast(T)\,.
\]
Let $\zeta\in \C^\times$ be a primitive $p^N$-th root of unity, for some $N$, and set: \[x_i:=u_i(\zeta)\in T.\]
In particular, $x_i$ is a semisimple $p$-power order element in $H$. The construction of these $x_i$ enables a simple description of their centralisers.

\begin{lemma} \label{lem:killnode}
For large enough $N$, the Dynkin diagram of $Z_H(x_i)_0$ is obtained from the Dynkin diagram of $H$ by deleting the $i$-th node.
\end{lemma}
\begin{proof}
The root system of $Z_H(x_i)_0$ is given by the subset
\[
\Phi_i=\{\alpha\in \Phi\mid \alpha(x_i)=1\}\,.
\]
For any root $\alpha=\sum_{j=1}^r m_j \alpha_j\in \Phi$ we have
\[
\alpha(x_i)=\zeta^{(\alpha,d_i\omega_i^\vee)}=\zeta^{d_i m_i}\, .
\]
Thus $\alpha(x_i)=1$ if and only if $p^N \mid m_i d_i$. Since $d_i\neq 0$, by choosing $N$ large enough we achieve $\alpha(x_i)=1$ if and only if $m_i=0$, i.e., if and only if $\alpha$ lies in the subroot system of $\Phi$ generated by the simple roots $\{\alpha_1,\dots,\alpha_{i-1},\alpha_{i+1},\dots,\alpha_r\}$. As there are only finitely many roots, we can find a uniform $N$ which works for every root. The Dynkin diagram of $\Phi_i$ is then obtained from that of $\Phi$ by deleting the $i$-th node.
\end{proof}
We invite the reader to compare the Weyl group orders and those of their node-deleted centralisers in Table 1.
\begin{table}[h]
\centering
\renewcommand{\arraystretch}{1.5}
\begin{tabular}{|c|c|c|c|}
\hline
\textbf{Type} & \textbf{Marked Dynkin diagram} & $|W(G')|$ & $|W(Z_{G'}(x)_0)|$\\
\hline
$A_n$ &
\begin{dyn}{72}{16}{-6}{-8}
  \bn{0}{0}\bn{14}{0}\elx{28}\bn{42}{0}\inb{56}{0}
  \drawline(0,0)(14,0)\drawline(14,0)(21,0)
  \drawline(35,0)(42,0)\drawline(42,0)(56,0)
\end{dyn}
& $(n+1)!$ & $n!$\\
\hline
$B_n$ &
\begin{dyn}{84}{16}{-6}{-8}
  \bn{0}{0}\bn{14}{0}\elx{28}\bn{42}{0}\bn{56}{0}\inb{70}{0}
  \drawline(0,0)(14,0)\drawline(14,0)(21,0)\drawline(35,0)(42,0)\drawline(42,0)(56,0)
  \drawline(56,2)(70,2)\drawline(56,-2)(70,-2)
  \drawline(66,3)(70,0)\drawline(66,-3)(70,0)
\end{dyn}
& $2^{n}\,n!$ & $n!$\\
\hline
$C_n$ &
\begin{dyn}{84}{16}{-6}{-8}
  \bn{0}{0}\bn{14}{0}\elx{28}\bn{42}{0}\bn{56}{0}\inb{70}{0}
  \drawline(0,0)(14,0)\drawline(14,0)(21,0)\drawline(35,0)(42,0)\drawline(42,0)(56,0)
  \drawline(56,2)(70,2)\drawline(56,-2)(70,-2)
  \drawline(60,3)(56,0)\drawline(60,-3)(56,0)
\end{dyn}
& $2^{n}\,n!$ & $n!$\\
\hline
$D_n$ &
\begin{dyn}{84}{38}{-6}{-19}
  \bn{0}{0}\bn{14}{0}\elx{28}\bn{42}{0}\bn{56}{0}
  \drawline(0,0)(14,0)\drawline(14,0)(21,0)\drawline(35,0)(42,0)\drawline(42,0)(56,0)
  \inb{70}{12}\bn{70}{-12}
  \drawline(56,0)(70,12)\drawline(56,0)(70,-12)
\end{dyn}
& $2^{n-1}\,n!$ & $n!$\\
\hline
$F_4$ &
\begin{dyn}{52}{16}{-6}{-8}
  \bn{0}{0}\bn{14}{0}\inb{28}{0}\bn{42}{0}
  \drawline(0,0)(14,0)\drawline(28,0)(42,0)
  \drawline(14,2)(28,2)\drawline(14,-2)(28,-2)
  \drawline(24,3)(28,0)\drawline(24,-3)(28,0)
\end{dyn}
& $2^{7}\cdot 3^{2}$ & $2\cdot 3\cdot 2$\\
\hline
$G_2$ &
\begin{dyn}{34}{16}{-4}{-8}
  \bn{18}{0}\inb{0}{0}
  \drawline(0,3)(18,3)\drawline(0,0)(18,0)\drawline(0,-3)(18,-3)
  \drawline(4,3)(0,0)\drawline(4,-3)(0,0)
\end{dyn}
& $3\cdot 2^{2}$ & $2$\\
\hline
$\widetilde{G_2}$ &
\begin{dyn}{52}{16}{-4}{-8}
  \drawline(0,3)(18,3)
  \drawline(0,0)(18,0)
  \drawline(0,-3)(18,-3)
  \drawline(4,3)(0,0)
  \drawline(4,-3)(0,0)

  \drawline(18,0)(34,0)

  \inb{0}{0}             
  \bn{18}{0}             
  \put(36,0){\circle{4}} 
\end{dyn}
& $3\cdot 2^{2}$ & $2\cdot 3$\\
\hline
$E_6$ &
\begin{dyn}{66}{38}{-6}{-19}
  \bn{0}{0}\bn{14}{0}\bn{28}{0}\bn{42}{0}\bn{56}{0}
  \drawline(0,0)(14,0)\drawline(14,0)(28,0)\drawline(28,0)(42,0)\drawline(42,0)(56,0)
  \inb{28}{12}
  \drawline(28,0)(28,12)
\end{dyn}
& $2^{7}\cdot 3^{4}\cdot 5$ & $6!$\\
\hline
$E_7$ &
\begin{dyn}{80}{38}{-6}{-19}
  \bn{0}{0}\bn{14}{0}\bn{28}{0}\bn{42}{0}\bn{56}{0}\bn{70}{0}
  \drawline(0,0)(14,0)\drawline(14,0)(28,0)\drawline(28,0)(42,0)
  \drawline(42,0)(56,0)\drawline(56,0)(70,0)
  \inb{28}{12}
  \drawline(28,0)(28,12)
\end{dyn}
& $2^{10}\cdot 3^{4}\cdot 5\cdot 7$ & $7!$\\
\hline
$E_8$ &
\begin{dyn}{94}{38}{-6}{-19}
  \bn{0}{0}\bn{14}{0}\bn{28}{0}\bn{42}{0}\bn{56}{0}\bn{70}{0}\bn{84}{0}
  \drawline(0,0)(14,0)\drawline(14,0)(28,0)\drawline(28,0)(42,0)\drawline(42,0)(56,0)
  \drawline(56,0)(70,0)\drawline(70,0)(84,0)
  \inb{28}{12}
  \drawline(28,0)(28,12)
\end{dyn}
& $2^{14}\cdot 3^{5}\cdot 5^{2}\cdot 7$ & $8!$\\
\hline
\end{tabular}
\caption{Deleting the circled node from the Dynkin diagram (or extended Dynkin diagram in the $(G_2,p=3)$-case) yields the Dynkin diagram of $Z_{G'}(x)_0$, for the semisimple $p$-power order element $x$ of Lemmas \ref{lem:killnode} and \ref{lem:g2}. The last column shows $|W(Z_{G'}(x)_0)|$, which is still divisible by every relevant prime $p$, aside from type $A_{p-1}$.}
\label{tbl:rootsystems}
\end{table}

Lemma \ref{lem:killnode} will handle almost every case we need, the only exception is type $G_2$ for $p=3$. Before proceeding with the general proof, we treat this case explicitly, noting that in this case the coweight and the cocharacter lattices agree.

\begin{lemma} \label{lem:g2}
Let $H$ be of type $G_2$ and $\zeta$ be a primitive third root of unity. Then the centraliser of $x:=\omega_1^\vee(\zeta)$ has root system of type $A_2$.
\end{lemma}
\begin{proof}
Let $\Phi_x\subset \Phi$ be the roots of $Z_H(x)_0$. A root $\alpha=m_1\alpha_1+m_2\alpha_2\in \Phi$ belongs to $\Phi_x$ if and only if $3\mid m_1$. The highest root is $\alpha_0=3\alpha_1+2\alpha_2$. So $\{-\alpha_0,\alpha_2\}$ is a set of simple roots for $\Phi_x$, and this is a root system of type $A_2$. It is obtained from the extended Dynkin diagram of $G_2$ by deleting the node corresponding to the short simple root $\alpha_1$, see Table \ref{tbl:rootsystems}.
\end{proof}

\begin{proposition} \label{prop:centraliser}
    If the adjoint form $G/Z(G)$ is not $PGL_p(\C)$, then there is a non-central semisimple $x\in G$ of $p$-power order such that $p$ divides the order of $W(Z_G(x)_0)$.
\end{proposition}
\begin{proof}
We have an isogeny $f\co Z(G)_0\times G'\to G$, and we know that $G'$ is simple and not of type $A_{p-1}$. Applying Lemma \ref{lem:killnode} to $H=G'$, aside from $G_2$ at $p=3$, we can find a non-central semisimple $y\in G'$ of $p$-power order such that $p$ divides the order of $W(Z_{G'}(y)_0)$ by the coweight corresponding to the marked nodes given in Table \ref{tbl:rootsystems}. Lemma \ref{lem:g2} supplies such a $y$ in the case of $G_2$ at the prime $p=3$. In every case, set $x:=f(y)\in G$. Then $x$ is semisimple and of $p$-power order, and because $f$ is an isogeny, the root system of $Z_G(x)_0$ agrees with that of $Z_{G'}(y)_0$. Therefore, $x$ is non-central and $p$ divides the order of $W(Z_G(x)_0)$.
\end{proof}

Because $x$ is non-central, $Z_G(x)_0$ is a proper subgroup of $G$, and since $G$ is connected, it has strictly smaller dimension. In view of this, by applying Propositions \ref{prop:smith}, \ref{prop:simple} and \ref{prop:centraliser}, we reduce to the case of adjoint type $PGL_p(\C)$.
\begin{proposition} \label{prop:basecase}
Let $G$ be a connected complex reductive group. If the adjoint form of $G'$ is $PGL_p(\C)$, then $\C^\times \times G$ is isomorphic to the product of a torus $Z$ with $SL_p(\C)$, $PGL_p(\C)$, or $GL_p(\C)$.
\end{proposition}

\begin{proof}
Consider the multiplication map $f\co Z(G)_0\times G'\to G$. The assumption on the adjoint form says that $G'$ is either $SL_p(\C)$ or $PGL_p(\C)$. If $f$ is an isomorphism, then we are in one of the first two cases, so we are done.

If the kernel of $f$ is nontrivial, it is the subgroup
\[
L:=\{(z^{-1},z)\mid z\in Z(G)_0\cap G'\}\subset Z(G)_0\times G'\,.
\]
Since $Z(G)_0\cap G'\subset Z(G')$, the nontriviality of $L$ implies $G'$ is isomorphic to $SL_p(\C)$ and $Z(G)_0\cap G'=Z(G')$. Fix such an isomorphism $G'\cong SL_p(\C)$, as well as a primitive $p$-th root of unity $\zeta\in \C^\times$, and an isomorphism $Z(G)_0\cong (\C^\times)^k$ where $k=\dim(Z(G)_0)$. Under these isomorphisms, $L$ identifies with the cyclic subgroup $F\subset (\C^\times)^k\times SL_p(\C)$ generated by $(\zeta^{a_1},\dots,\zeta^{a_k},\zeta I)$ for a unique non-zero vector $a=(a_1,\dots,a_k)\in (\Z/p)^k$.

Replacing $G$ by $\C^\times\times  G$, we may assume that $k\geq 2$. Then $SL_k(\Z/p)$ acts transitively on the non-zero vectors in $(\Z/p)^k$, and since the reduction $SL_k(\Z) \to SL_k(\Z/p)$ is surjective \cite[Lemma 1.38]{Sh71}, there exists an $M\in SL_k(\Z)$ with $Ma=(0,\dots,0,-1)$ modulo $p$. The matrix $M$ defines an automorphism of $(\C^\times)^k$ and, extending it by the identity on $SL_p(\C)$, an automorphism of $(\C^\times)^k\times SL_p(\C)$. This automorphism then maps $F$ onto the subgroup $F'$ generated by $(1,\dots,1,\zeta^{-1},\zeta I)$. It follows that there is an isomorphism
\[
\C^\times \times G \cong ((\C^\times)^k\times SL_p(\C))/F' \cong (\C^\times)^{k-1} \times GL_p(\C)\,,
\]
which proves the assertion.
\end{proof}

\section{The base case} \label{sec:basecase}

In this section we prove Theorem \ref{thm:mainreductive} by verifying the base case of type $A_{p-1}$ through explicit calculation. Following Proposition \ref{prop:basecase}, since we have an isomorphism \[C_m(Z\times G)_1\cong Z^m\times C_m(G)_1\,,\] and the cohomology of $\mathbb{C}^\times\simeq S^1$ is torsion free, we reduce to showing $p$-torsion in the cohomology of $C_m(G)_1$ for \[G=GL_p(\C),SL_p(\C),PGL_p(\C).\] Since $C_2(G)_1$ is a retract of $C_m(G)_1$ for higher $m$, we may further assume that $m=2$. We will verify these cases working with the associated compact Lie groups, using Theorem \ref{thm:pettetsouto}, so we need to find $p$-torsion in the space $C_2(K)_1$ for \[K=U(p),\,SU(p),\,PU(p).\] First, let us reduce to the case of $U(p)$.

\begin{proposition}\label{prop:U(p) reduction}
If $C_2(U(p))_1$ has $p$-torsion in its integral cohomology, then so do the spaces $C_2(SU(p))_1$ and $C_2(PU(p))_1$.
\end{proposition}

\begin{proof}
Working with $\Z_{(p)}$-coefficients for cohomology, we will show that if either $C_2(SU(p))_1$ or $ C_2(PU(p))_1$ has no $p$-torsion, then $C_2(U(p))_1$ also has no $p$-torsion. Consider the following fibrations induced by central quotient and the determinant respectively:
\begin{alignat*}{2}
U(1)^2& \to C_2(U(p))_1\to C_2(PU(p))_1\,, \\
C_2(SU(p))_1& \to C_2(U(p))_1\to S^1\times S^1\, .
\end{alignat*}
The quotient of $C_2(U(p))_1$ by the (free) central $U(1)^2$-action is $C_2(PU(p))_1$, because any commuting pair in the identity component $C_2(PU(p))_1$ is conjugate to one in the maximal torus of $PU(p)$. Similarly, the fibre of this second fibration is connected, as any commuting pair in the simply connected group $SU(p)$ lies in some maximal torus.

The rational equivalence of Proposition \ref{prop:desing} and the rank calculation of Corollary \ref{cor:constantrank} show that over the rationals, the Serre spectral sequences for these fibrations collapse at the $E_2$-page, and we have trivial monodromy. If both the base and the fibre have no $p$-torsion in their $\mathbb{Z}_{(p)}$-cohomology, and the Serre spectral sequence collapses over $\Q$ with trivial monodromy, we may deduce collapse of the Serre spectral sequence over $\Z_{(p)}$. This implies there is no $p$-torsion in the cohomology of the total space, completing the proof.
\end{proof}

It thus remains to exhibit torsion in the cohomology of $C_2(U(p))_1$. Our proof is based on the following spectral sequence degeneration criterion.

\begin{proposition} \label{prop:ss}
Let $F\to E\to B$ be a fibre bundle over a connected base, whose fibre $F$ is a finite CW-complex of dimension $f$. Assume there exist nested closed sub-bundles \[E_{\leq 0}\subset E_{\leq 1}\subset \cdots \subset E_{\leq f}=E\] such that the fibres $F_{\leq k}$ of $E_{\leq k}$ are finite CW-complexes of dimension $\leq k$. If for some coefficients $R$, and all $0\leq k\leq f$ the induced maps\[H^i(F,R)\to H^i(F_{\leq k},R)\] are isomorphisms in all degrees $i \leq k$, then the Serre spectral sequence degenerates at $E_2$ and the associated filtration splits. In particular, we have an isomorphism\[H^n(E,R)\cong \bigoplus_{p+q=n} H^p(B,\mathcal{H}^q(F,R))\]
for every $n\geq 0$.
\end{proposition}

\begin{proof}
We use $R$-coefficients throughout the proof. First we prove collapse of the spectral sequence for $E_{\leq k}$ for every $k$. We induct over the finite sequence of sub-bundles $E_{\leq k}\subset E$, noting that the base case $E_{\leq 0}$ follows at once as the $E_2$-page lies on a single line.

Assume the statement holds for $E_{\leq k-1}$. The inclusion of bundles \[E_{\leq k-1}\to E_{\leq k}\] induces a morphism of Serre spectral sequences \[\phi:E^{p,q}_{\bullet}\to \widehat{E}^{p,q}_{\bullet}.\] By induction, we know that all the differentials in the target spectral sequence are zero. Our assumption on the fibre cohomology implies that, for $q\leq k-1$, the induced map of $R$ local systems on $B$ \[\mathcal{H}^q(F_{\leq k})\to \mathcal{H}^q(F_{\leq k-1})\] is an isomorphism, and $\mathcal{H}^q(F_{\leq k})=0$ for $q> k$. Consider the first page $E_l$, $l\geq 2$, with a potentially nonzero differential, which is therefore of the form \[d:E_l^{p,k}\to E_l^{p+l,k+1-l}.\]
By naturality of the Serre spectral sequence, we have a commutative diagram induced by $\phi$:
\begin{equation}\label{Diagram:nat of SS}
    \begin{tikzcd}
E^{p,k}_{l}\arrow[d,"d"]\arrow[rr,"\phi_l^{p,k}"]&& \widehat{E}^{p,k}_{l}\arrow[d,"0"]\\
E^{p+l,k+1-l}_{l}\arrow[rr,"\phi_l^{p+l,k+1-l}"]&& \widehat{E}^{p+l,k+1-l}_{l}.
\end{tikzcd}
\end{equation}
By the minimality of $l\geq 2$, we see that $\phi_l^{p+l,k+1-l}$ equals $\phi_2^{p+l,k+1-l}$, which we may identify with the canonical map
 \[H^{p+l}(B,\mathcal{H}^{k+1-l}(F_{\leq k}))\to H^{p+l}(B,\mathcal{H}^{k+1-l}(F_{\leq k-1}))\, .\]
This map is an isomorphism, since $k+1-l\leq k-1$. By commutativity of (\ref{Diagram:nat of SS}), we deduce that $d$ must be zero, and so the spectral sequence collapses at $E_2$.

Now we prove the second statement of the proposition, namely that the associated filtration splits. For $0\leq k \leq f$, let
\[
0\subset G^nH^n(E_{\leq k})\subset \cdots \subset G^{1}H^n(E_{\leq k})\subset G^{0}H^n(E_{\leq k})=H^n(E_{\leq k})
\]
denote the finite, descending filtration of $H^n(E_{\leq k})$ to which the Serre spectral sequence converges. For every $p\geq 0$ we have
 \[G^pH^n(E_{\leq k})/G^{p+1}H^n(E_{\leq k})\cong E^{p,n-p}_\infty=E_2^{p,n-p}\cong H^p(B,\mathcal{H}^{n-p}(F_{\leq k})).\]
By considering the associated graded, we see that the map \[G^pH^n(E_{\leq k})\to G^pH^n(E_{\leq k-1})\] is an isomorphism if $p\geq n-k+1$.

Now we show that for every $0\leq p\leq n-1$ the inclusion
\[
G^{p+1}H^n(E)\to G^{p}H^n(E)
\]
is split. By naturality of the Serre spectral sequence, the bundle maps $E_{\leq n-p-1}\to E_{\leq n-p}\to E$ (set $E_{\leq k}:=E$ if $k\geq f$) induce a commutative diagram:
 \[\begin{tikzcd}
    G^{p+1}H^n(E)\arrow[d,"\simeq"]\arrow[r]&G^{p}H^n(E)\arrow[d,"\simeq"]\\
     G^{p+1}H^n(E_{\leq n-p})\arrow[r]\arrow[d,"\simeq"]& G^{p}H^n(E_{\leq n-p})\arrow[d]\\
     G^{p+1}H^n(E_{\leq n-p-1})\arrow[r,"\simeq"]& G^{p}H^n(E_{\leq n-p-1})\, .&
 \end{tikzcd}\]
The lower horizontal map is an isomorphism, because the cohomology of the fibre $F_{\leq n-p-1}$ vanishes in degrees $\geq n-p$. The composition provides the desired splitting.
\end{proof}

\begin{remark}
An abstract reason for this degeneration may be given as follows. Since our spaces are suitably nice, the Serre spectral sequence for the fibration \[p:E\to B\] agrees with the Leray spectral sequence in sheaf cohomology. The Leray spectral sequence is induced by the $t$-structure ascending filtration $\tau_{\leq\bullet } Rp_* \underline{R}_E$ of the total derived functor $Rp_*$ applied to the constant sheaf $\underline{R}_E$ of $R$ modules on $E$. Our nested sub-bundles \[p_{\leq k}:E_{\leq k}\to B\] give another filtration of this object by $Rp_{\leq k\ast}\underline{R}_E$ with the same subquotients. These two homotopy coherent filtrations are totally opposed, and this yields the desired degeneration and splitting.
\end{remark}

\begin{remark}
In the case when each quotient $F_{\leq k}/F_{\leq k-1}$ is a wedge of $k$-spheres, our assumptions imply the associated cellular chain complex for $F$ has trivial differentials. Proposition \ref{prop:ss} may be viewed as a $B$-parametrised variant of this situation.
\end{remark}

\begin{corollary}\label{cor:degen U(p)}
Let $U(p)/T\cdot S_p$ denote the unordered flag manifold. Then the Serre spectral sequence for the fibre bundle \[U(p)/T\times_{S_p} T^2\to U(p)/T\cdot S_p\] degenerates at the $E_2$ page for any choice of coefficients, and the associated filtration splits.
\end{corollary}

\begin{proof}
Taking $T$ to be the diagonal unitary matrices, we see that the $S_p$ monodromy action is given by permutation of diagonal entries. Taking the two cell CW-structure on $S^1$, and the product CW-structure on $T^2\cong (S^1)^p\times (S^1)^p$, we see that the skeleta $(T^2)_{\leq k}$ of $T^2$ are $S_p$-invariant and we may define the sub-bundles of $E=U(p)/T\times_{S_p} T^2$ by \[E_{\leq k}:=U(p)/T\times_{S_p} (T^2)_{\leq k}\, . \]
The CW-skeleta of the fibres have trivial cellular boundary maps for any choice of coefficients, so the assumptions of the proposition are satisfied, and we have the desired degeneration.
\end{proof}

We will leverage the degeneration of the spectral sequence to find torsion in the cohomology of $C_2(U(p))_1$, coming from torsion in the cohomology of $U(p)/T\cdot S_p$.

\begin{proposition}
The space $C_2(U(p))_1$ has $p$-torsion in its cohomology in degree $p^2+2$.
\end{proposition}

\begin{proof}
  Following Proposition \ref{prop:desing}, we have the desingularisation map \[\pi:U(p)/T\times_{S_p} T^2\to C_2(U(p))_1\, .\]
 Let $Z\subset T$ be the centre of $U(p)$. Within our domain space \[X:=U(p)/T\times_{S_p} T^2\] we have the subspace $A$ of elements with central coordinates in $T^2$\[A:=U(p)/T\times_{S_p} Z^2,\] and its image $\overline{A}$ under $\pi$: \[\overline{A}=Z^2.\]
 Taking quotients, we have the following commuting diagram:\[\begin{tikzcd}
 X\arrow[r,"\pi"]\arrow[d]&C_2(U(p))_1\arrow[d]\\
 X/A\arrow[r,"\tilde{\pi}"]&C_2(U(p))_1/\overline{A}
 \end{tikzcd}\]

 The map $\tilde{\pi}$ has $\Z_{(p)}$-acyclic fibres; following \cite{Baird}, the fibre of the map ${\pi}$ over a pair of commuting elements $(x_1,x_2)$ identifies with the space of orthogonal line decompositions refining the simultaneous eigenspace decomposition of $x_1$ and $x_2$. In particular, away from $Z^2$ these are products of the spaces $U(k)/T\cdot S_k$ for $k<p$, and are thus $\mathbb{Z}_{(p)}$-acyclic. The map $\tilde{\pi}$ is therefore a $\Z_{(p)}$-cohomology isomorphism by the Vietoris-Begle theorem.
 
 The spaces $A$ and $\overline{A}$ have dimensions $p^2-p+2$ and $2$ respectively. Therefore, by the long exact sequence in relative cohomology, $X$ and $C_2(U(p))_1$ have the same $\mathbb{Z}_{(p)}$-cohomology in degrees at least $p^2-p+4$.
 
 It remains to supply torsion classes in the $\mathbb{Z}_{(p)}$-cohomology of $X$. For this, we use the degeneration of the Serre spectral sequence of Corollary \ref{cor:degen U(p)} in bidegree $(p^2-2p+2,2p)$. The $\F_p$-cohomology of the base $U(p)/T\cdot S_p$ is one dimensional in degree $p^2-2p+2$, and this degree is maximal, by \cite[Corollary 6.9]{GuerraJanna}. The Bockstein sequence shows that a generator lifts to a class in $H^{p^2-2p+2}(U(p)/T\cdot S_p,\Z_{(p)})$, which must be torsion, because $\tilde{H}^{*}(U(p)/T\cdot S_p,\mathbb{Q})=0$.
 
 The corresponding local system $\mathcal{H}^{2p}(T^2,\Z_{(p)})$ is trivial, because it is the square of the sign representation of $S_p$.  Therefore, our argument supplies a nonzero $p$-torsion class in \[H^{p^2-2p+2}(U(p)/T\cdot S_p,H^{2p}(T^2;\Z_{(p)}))\] 
 which, by Corollary \ref{cor:degen U(p)}, is a direct summand of $H^{p^2+2}(X,\Z_{(p)})$.
\end{proof}

\begin{remark}
The cohomology of the space $U(p)/T\cdot S_p$ is rationally trivial, but has large amounts of torsion \cite{GuerraJanna}. This is the essential source of the torsion showing up in our proof. It would be interesting to have an algebro-geometric explanation for this large amount of torsion.
\end{remark}

\section{The Lie algebra case}

In this section we indicate how the proof adapts to the setting of commuting $m$-tuples in the Lie algebra of a compact Lie group, thus proving Theorem \ref{thm:mainliealgebra}.

Let $K$ be a connected compact Lie group with Lie algebra $\mathfrak{k}$. Let $W$ be the Weyl group of $K$. The following statement from \cite[Theorem D]{AGG} is the analogue in the Lie algebra setting of Corollary \ref{cor:constantrank}.

\begin{proposition}
    Let $k$ be a field of characteristic coprime to $|W|$. Then the one-point compactification $C_m(\mathfrak{k})^+$ is a $k$-homology sphere. In particular, for any $x\in K$ we have an equality of ranks
    \[
    \dim_k H^\ast_c(C_m(\mathfrak{k}),k) = \dim_k H^\ast_c(C_m(\mathfrak{z}_{\mathfrak{k}}(x)),k)\, ,
    \]
    where $\mathfrak{z}_\mathfrak{k}(x)=\{X\in \mathfrak{k}\mid \mathrm{Ad}(x)X=X\}$ is the Lie algebra of $Z_K(x)$.
\end{proposition}

Combining this with Lemma \ref{lem:recognition} (applied to $C_m(\mathfrak{k})^+$) we obtain:

\begin{corollary}
    If $p$-torsion occurs in the compactly supported cohomology of $C_m(\mathfrak{k})$, then $p$ divides the order of $W$.
\end{corollary}

Applying Floyd's inequality to the adjoint action of $K$ on $C_m(\mathfrak{k})^+$, the same reduction steps as in the Lie group case reduce the proof of Theorem \ref{thm:mainliealgebra} to:

\begin{proposition}
    For every prime $p$ and every integer $m\geq 2$, the space $C_m(\mathfrak{su}_p)$ has $p$-torsion in its compactly supported integral cohomology.
\end{proposition}
\begin{proof}
    Let $T\subset SU(p)$ be a maximal torus and $\mathfrak{t}$ its Lie algebra. As an $S_p$-representation $\mathfrak{t}$ is the $(p-1)$-dimensional standard representation of $S_p$. Let $S(\mathfrak{su}_p^m)\subset \mathfrak{su}_p^m$ and $S(\mathfrak{t}^m)= \mathfrak{t}^m\cap S(\mathfrak{su}_p^m)$ denote the unit spheres with respect to an Ad-invariant inner product. Intersecting $C_m(\mathfrak{su}_p)$ with $S(\mathfrak{su}_p^m)$ to obtain $C_m^1(\mathfrak{su}_p)$, the Lie algebra variant of the desingularisation map, \[SU(p)/T\times_{S_p} S(\mathfrak{t}^m)\to C_m^1(\mathfrak{su}_p)\,,\] is a $\Z_{(p)}$-cohomology isomorphism by the Vietoris-Begle theorem (cf. \cite[Lemma A.13]{AGG}).
    As $C_m(\mathfrak{su}_p)^+$ is the (unreduced) suspension of $C^1_m(\mathfrak{su}_p)$, it suffices to show that the total space of the sphere bundle
    \[
    E:=SU(p)/T\times_{S_p} S(\mathfrak{t}^m) \to SU(p)/T\cdot S_p
    \]
    has $p$-torsion in its singular $\Z_{(p)}$-cohomology.
    
    The $E_2$-page of the Serre spectral sequence is concentrated in two rows of vertical degrees $0$ and $\dim S(\mathfrak{t}^m)=m(p-1)-1$. The local system $\mathcal{H}^{m(p-1)-1}(S(\mathfrak{t}^m),\Z_{(p)})$ is the trivial module or the sign module, depending on whether $m$ is even or odd. In either case, a non-zero differential must have target in horizontal degrees at least $m(p-1)$. 
    
    If $p$ is odd, then by \cite[Corollary 6.9]{GuerraJanna} the base space $SU(p)/T\cdot S_p$ has $p$-torsion in degree $2p-3< m(p-1)$, which thus survives the spectral sequence and gives a $p$-torsion class in $G^{2p-3}H^{2p-3}(E,\Z_{(p)})\subset H^{2p-3}(E,\Z_{(p)})$.
    
    If $p=2$, then the base space is just $\R P^2$. Now if $m\geq 3$, then the $2$-torsion class in $H^2(\R P^2,\Z_{(2)})$ survives, while if $m=2$, then the local system is trivial and the $2$-torsion class in bidegree $(2,1)$ survives. 
\end{proof}
Theorem \ref{thm:pettetsouto}, identifying the cohomology of $C_m(G)_1$ and $C_m(K)_1$, fails in the Lie algebra case. In fact, for a complex reductive Lie algebra $\mathfrak{g}$, the total rank of the compactly supported rational cohomology of $C_m(\mathfrak{g})$ does not depend only on the rank of $\mathfrak{g}$:

\begin{proposition} \label{prop:complexreductiveliealgebra}
Consider the rank $2$ complex Lie algebras $\mathfrak{sp}_4^\mathbb{C}$ and $\mathfrak{gl}_2^\mathbb{C}$. Then the compactly supported cohomology of their varieties of commuting pairs have different ranks:
\begin{alignat*}{1}
\dim_\mathbb{Q}H^*_c(C_2(\mathfrak{gl}_2^\mathbb{C}),\Q) & =3\,, \\
\dim_\mathbb{Q}H^*_c(C_2(\mathfrak{sp}_4^\mathbb{C}),\Q) & \geq 9\, .
\end{alignat*}
\end{proposition}

\begin{proof}
Since $\mathfrak{gl}^\C_2$ splits as $\mathfrak{z}\oplus \mathfrak{sl}_2^\C$, where $\mathfrak{z}=\C$ is the centre, there is an isomorphism $C_2(\mathfrak{gl}_2^\C) \cong \C^2\times C_2(\mathfrak{sl}_2^\C)$. The space $C_2(\mathfrak{sl}_2^\C)$ has compactly supported rational cohomology of total rank $3$; this can be computed from the projection \[C_2(\mathfrak{sl}^\C_2)\to \mathfrak{sl}^\C_2,\quad (X,Y)\mapsto X\, .\]
The fibre over any non-zero $X$ is $\mathfrak{z}_{\mathfrak{sl}_2^\C}(X)\cong \C$ (in $\mathfrak{sl}_2^\C$ any non-zero element is regular), while the fibre over $0$ is $\mathfrak{sl}_2^\C\cong \C^3$. This gives a stratification of $C_2(\mathfrak{sl}_2^\C)$ with one open stratum $U\cong (\mathfrak{sl}_2^\C\backslash \{0\})\times \C$ and a closed complement $\C^3$, hence a long exact sequence
\[
\cdots \to H_c^k(U) \to H_c^k(C_2(\mathfrak{sl}_2^\C)) \to H^k_c(\C^3)\to H^{k+1}_c(U) \to \cdots
\]
The cohomology of $U$ is in degrees $3,8$, and the cohomology of $\C^3$ is in degree $6$; the total rank is $3$, which proves the first part.

The space $C_2(\mathfrak{sp}^\C_4)$ has much more cohomology, however. We obtain our lower bound by counting points over finite fields (see the work of Hausel and Rodriguez-Villegas \cite{HauselRodriguez} for a spectacular application of this method in the context of character varieties). The complex variety $C_2(\mathfrak{sp}_4^\mathbb{C})$ is defined over $\mathbb{Z}$. Following Fulman-Guralnick \cite[Theorem 4.6]{FulmanGuralnick}, the number of solutions to the defining equations of the scheme $C_2(\mathfrak{sp}_4)$ in a finite field $\F_q$ of odd characteristic is given by
\[\#C_2(\mathfrak{sp}_4)(\mathbb{F}_q)=q^{12}+2q^{11}+q^{10}-q^8-2q^7-q^6+q^4\,.\]

Let $p$ be an odd prime, $q=p^r$ some power of $p$, and $\ell$ a prime not equal to $p$, and let $X:=C_2(\mathfrak{sp}_4)\times_{\Z} \bar{\F}_p$. By the Grothendieck trace formula \cite{De77}, $\#C_2(\mathfrak{sp}_4)(\mathbb{F}_q)$ is the supertrace of the $r$-th power of the geometric Frobenius $F$ on the compactly supported $\ell$-adic cohomology of $X$:
\[
\sum_{k\geq 0} (-1)^k \mathrm{Tr}(F^r \mid H^k_{\acute{e}t,c}(X,\Q_\ell)) =\#C_2(\mathfrak{sp}_4)(\mathbb{F}_q)\, .
\]
As this equality holds for any $r$, we can deduce that 
\[
p^{12},p^{11},p^{10},p^8,p^7,p^6,p^4
\]
are eigenvalues of $F$ on $\bigoplus_{k\geq 0} H^k_{\acute{e}t,c}(X,\Q_\ell)$ with signed multiplicities $(1,2,1,-1,-2,-1,1)$. This implies that
\[
\dim_{\Q_\ell} H^\ast_{\acute{e}t,c}(X,\Q_\ell)\geq 1+2+1+1+2+1+1=9\,.
\]
For all but finitely many primes $p$, the compactly supported $\ell$-adic cohomology of $X$ agrees with that of the complex variety $C_2(\mathfrak{sp}_4^\C)$, and the étale to analytic comparison map (see \cite[Ch.~III,\S3]{Mi80}) then supplies an isomorphism with singular cohomology, so that altogether
\[H^*_{\acute{e}t,c}(X,\Q_{\ell})\cong H^*_{c}(C_2(\mathfrak{sp}^\C_4),\Q_{\ell})\, .\]
We conclude that $\dim_\Q H^\ast_c(C_2(\mathfrak{sp}_4^\C),\Q)=\dim_{\Q_\ell} H^\ast_c(C_2(\mathfrak{sp}_4^\C),\Q_{\ell})\geq 9$, which is the desired lower bound.
\end{proof}

The question of what the torsion primes are for commuting varieties in a complex reductive Lie algebra remains open.

\bibliographystyle{amsalpha}
\bibliography{refs}

\end{document}